\documentclass[12pt,a4paper,reqno]{amsart}
\usepackage{amsmath,amssymb,ifthen}
\usepackage{graphicx,psfrag,subfigure}
\usepackage{cleveref}
\usepackage{fullpage}

\def\R{{\mathbb R}}
\def\C{{\mathbb C}}
\def\N{{\mathbb N}}

\def\ba{\boldsymbol{a}}
\def\bb{\boldsymbol{b}}
\def\bc{\boldsymbol{c}}
\def\bv{\boldsymbol{v}}

\def\bw{\boldsymbol{w}}
\def\bm{\boldsymbol{m}}

\def\bT{\boldsymbol{T}}

\def\bn{\boldsymbol{n}}

\def\II{{\mathcal I}}

\def\RR{{\mathcal R}}

\def\norm#1#2{\|#1\|_{#2}}

\def\set#1#2{\big\{#1\,:\,#2\big\}}

\def\eps{\varepsilon}
\def\normL2#1#2{\|#1\|_{L^2(#2)}}

\newcommand{\dual}[3][]{#1\langle#2\,,\,#3#1\rangle}

\def\br{\boldsymbol{r}}
\def\bR{\boldsymbol{R}}

\newtheorem{theorem}{Theorem}

\newtheorem{lemma}[theorem]{Lemma}

\newtheorem{definition}[theorem]{Definition}

\begin{document}%%%%%%%%%%%%%%%%%%%%%%%%%%%%%%%%%%%%%%%%%%%%%%%%%%%%%
\title{Existence of arbitrarily smooth solutions of the LLG equation in 3D with natural boundary conditions}

\author{Michael Feischl and Thanh Tran}

\begin{abstract}
We prove that the Landau-Lifshitz-Gilbert equation in three space dimensions with homogeneous Neumann boundary conditions
admits arbitrarily smooth solutions, given that the initial data is sufficiently close to a constant function.
\end{abstract}

\maketitle

\section{Introduction}
The Landau-Lifshitz-Gilbert (LLG) equation   is widely considered as a valid model of micromagnetic
phenomena occurring in, e.g., magnetic sensors, recording heads,
and magneto-resistive storage device~\cite{Gil,LL,prohl2001}. It describes the  precessional motion of magnetization in ferromagnets. The main difficulty of the LLG equation is its strongly non-linear character.

Classical results concerning existence and non-uniqueness of solutions can
be found in~\cite{as,vis}. The existence of weak solutions is proved for 2D and 3D in~\cite{alouges}. It is known that weak solutions are in general not unique but exist globally.
Throughout the literature, there are various works on weakly-convergent numerical
approximation methods for the LLG (coupled to the Maxwell-equations)
equations~\cite{alouges,alok,bako,bapr,cim,ellg,thanh} (the list is not
exhausted) even without an artificial projection step~\cite{michele,LLGnum}.

This paper considers the question of existence of arbitrarily smooth strong solutions of this equation.
For the case of the 2D torus, the book~\cite{prohl2001} gives an exhaustive overview on results concerning the existence and regularity of strong solutions. A brief summary of the state of the art for 2D domains
with periodic boundary conditions could be phrased as follows: There exist arbitrarily smooth solutions provided that the initial data is sufficiently close to a constant function. Moreover, there exist arbitrarily smooth local-in-time solutions for 
initial data of finite energy (see, e.g.,~\cite{2dexist}). For the 3D case, much less is known in terms of strong solvability. For the 3D torus (with periodic boundary conditions)~\cite{cimExist} proves $H^2$-regularity local in time
for the coupled system of LLG and Maxwell-equations. The work~\cite{3dell} proves global existence of strong solutions for small initial energies on small ellipsoids. The survey
article~\cite{harmonicmap} summarizes results in the context of the evolution of
harmonic maps (which however does not cover the LLG equation). A recent
paper~\cite{Melcher12} studies the existence, uniqueness and asymptotic behavior
of solutions in the whole spatial space~$\R^3$.

To the authors best knowledge, this work is the first which proves existence of arbitrarily smooth (non-trivial) solutions on bounded 3D domains.
It also gives a first result on existence of arbitrarily smooth strong solutions with natural boundary conditions (in 2D and 3D). It is worth mentioning that the proof is constructive in 
the sense that a convergent sequence of approximate solutions is designed algorithmically. The limit of this sequence turns out to be a smooth strong solution of the LLG equation.

The main motivation to prove existence of smooth strong solutions for the LLG equation originated in the recent work~\cite{LLGnum} by the authors. There, we proved a~priori error estimates
for a time integrator for the LLG equations (as well as the coupled LLG-Maxwell system) which imply strong convergence of the numerical method in case of smooth strong solutions.
Thus, the present work justifies the assumptions in~\cite{LLGnum}.

\section{The Landau-Lifshitz-Gilbert equation}
Consider a bounded {smooth} domain $D\subset \R^3$ with connected
boundary $\Gamma$ having the outward normal vector $\bn$. Note that all the results in this paper also hold true for $D\subset \R^n$, $n\geq 2$. For brevity of presentation, however,
we only consider the physically most relevant case $n=3$.
We define $D_T:=(0,T)\times D$ and $\Gamma_T := (0,T)\times\Gamma$ for $T>0$.
 We start with the LLG equation which reads as
 \begin{align}
  \bm_t - \alpha\bm\times\bm_t &= -C_e\bm \times \Delta\bm
 \quad\text{in }D_T\label{eq:llg}
\end{align}
for some constant $C_e>0$.
Here the parameter {$\alpha$ is a positive constant.}
It follows from~\cref{eq:llg} that $|\bm|$ is constant. We
follow the usual practice to normalize~$|\bm|$.
The following conditions are imposed on the solution of~\cref{eq:llg}:
\begin{subequations}\label{eq:con}
\begin{alignat}{2}
\partial_n\bm&=0
&& \quad\text{on }\Gamma_T,\label{eq:con1} 
\\
|\bm| &=1
&& \quad\text{in } D_T, \label{eq:con2}
\\
\bm(0,\cdot) &= \bm^0
&& \quad\text{in } D, \label{eq:con3}
\end{alignat}
\end{subequations}
where $\partial_n$ denotes the normal derivative.

The initial data~$\bm^0$ satisfies~$|\bm^0|=1$ in~$D$.
The condition~\cref{eq:con2} together with basic
properties of the cross product leads to the following equivalent
formulation of~\cref{eq:llg}:
\begin{align}\label{eq:llg2}
\alpha\bm_t+\bm\times\bm_t= C_e\Delta \bm-C_e(\bm\cdot \Delta\bm)\bm
\quad\text{in } D_T.
\end{align}

Before stating the main result of the article, we set some notations.
Bold letters (e.g.~$\bv$) will be used for vector functions.
However, as there is no confusion, we still use~$L^2(D_T)$ to denote the Lebesgue
space of vector functions taking values in~$\R^3$, i.e., we will
write~$\bv\in L^2(D_T)$ instead of~$\bv\in L^2(D_T)^3$. 
The same rule applies to other function spaces.

The following function spaces will be frequently used. For 
any non-negative integer~$k\in\N_0=\{0,1,2,\ldots\}$, we define
\[
H^{k,2k}(D_T):=\set{v\in L^2(D_T)}{\norm{v}{H^{k,2k}(D_T)}<\infty}
\]
where the norm is defined by
\begin{align*}
 \norm{v}{H^{k,2k}(D_T)}
:=
\sum_{\ell=0}^k\norm{\partial_t^\ell v}{L^2(0,T;H^{2k-2\ell}(D))}.
\end{align*}
The corresponding seminorm is
\begin{align*}
  |v|_{H^{k,2k}(D_T)}:=\Big(\sum_{\ell=1}^{2k}\norm{D^\ell
v}{L^2(D_T)}^2\Big)^{1/2}+\sum_{\ell=1}^k\norm{\partial_t^\ell
v}{L^2(0,T;H^{2k-2\ell}(D))},
\end{align*}
where $D^\ell$ denotes $\ell^{\rm{th}}$-order partial derivatives with
respect to the spatial variables.

Finally, we define
\begin{equation}\label{eq:H1 sta}
H^1_\star(D)
:=
\set{\bv\in H^1(D)}{\Delta\bv\in L^2(D) 
\quad\text{and}\quad \partial_n\bv = 0 \text{ on } \Gamma}.
\end{equation}

We are now ready to state the main result of the paper.
\begin{theorem}\label{thm:main}
Assume that the initial data~$\bm^0$ satisfies $|\bm^0|=1$ in~$D$
and, for some integer $k\ge3$,
\begin{enumerate}
\renewcommand{\labelenumi}{(\roman{enumi})}
\item
$\bm^0\in H^{2k}(D)\cap H^1_\star(D)$;
\item
$D^{j}\bm^0\in H^1_\star(D)$ for all $j/2\leq k-1$;
\item
$|\bm^0|_{H^{2k}(D)}$ is sufficiently small.
\end{enumerate}
Then the problem~\cref{eq:llg}--\cref{eq:con} has a smooth strong solution 
$\bm\in H^{k,2k}(D_T)$ which satisfies
\begin{align}\label{eq:normsmooth}
	\norm{\bm}{H^{k,2k}(D_T)}\leq C_{\rm smooth} \norm{\bm^0}{H^{2k}( D)},
\end{align}
where $C_{\rm smooth}>0$ depends only on $\alpha$, $C_e$, $T$, and $k$.
\end{theorem}

\section{Auxiliary Results}
For the reader's convenience, we state in the following lemma
some well-known results regarding Sobolev embeddings and traces.
\begin{lemma}\label{lem:sobolev}
\mbox{}
\begin{enumerate}
\renewcommand{\labelenumi}{(\roman{enumi})}
\item The embeddings $H^1(D)\hookrightarrow L^6(D)$ as well as $H^{1,2}(D_T)\hookrightarrow L^2(0,T;L^\infty(D))\cap L^\infty(0,T;L^2(D))$ are continuous.
\item
The embedding $H^{k+2,2k+4}(D_T)\hookrightarrow W^{k,\infty}(D_T)$
is continuous for all $k\in\N_0$.
\item
If $\bw\in H^{k,2k}(D_T)$ for~$k\ge 1$ then
$\partial_t^iD^{j}\bw(0)\in H^1(D)$ for all $i+j/2\le k-1$.
\end{enumerate}
\end{lemma}
\begin{proof}
We first prove~(i). The embedding $H^1(D)\hookrightarrow L^6(D)$ follows from the standard Sobolev inequality. By definition of $H^{1,2}(D_T)$, there holds
\begin{align*}
 H^{1,2}(D_T):= H^1(0,T;L^2(D))\cap L^2(0,T;H^2(D)).
\end{align*}
The well-known embeddings $H^1(0,T;L^2(D))\hookrightarrow L^\infty(0,T;L^2(D))$ and $L^2(0,T;H^2(D))\hookrightarrow L^2(0,T;L^\infty(D))$ (since $D\subset \R^3$)
conclude~(i).

Second, we prove~(ii).
Since $D\subset\R^3$, it is well-known that the embeddings 
\[
H^1(0,T;H^{\ell+2}(D))
\hookrightarrow 
H^1(0,T;W^{\ell,\infty}(D))
\hookrightarrow 
L^{\infty}(0,T;W^{\ell,\infty}(D))
\]
are continuous for any~$\ell\ge0$; see e.g.~\cite{lions}.
On the other hand, we can write
\[
W^{k,\infty}(D_T)
=
\set{v}{\partial_t^i v\in L^{\infty}(0,T;W^{k-i,\infty}(D)),
\quad i=0,\ldots,k}.
\]
Hence the embedding
\[
\set{v}{\partial_t^i v\in H^{1}(0,T;H^{k-i+2}(D)),
\quad i=0,\ldots,k}
\hookrightarrow
W^{k,\infty}(D_T)
\]
is continuous. Consequently, the embedding
\[
\bigcap_{i=0}^k H^{i+1}(0,T;H^{k-i+2}(D))
\hookrightarrow
W^{k,\infty}(D_T)
\]
is continuous. 
Since~$H^{k+2,2k+4}(D_T)\subset\bigcap_{i=0}^k H^{i+1}(0,T;H^{k-i+2}(D))$, part~(ii) is proved.

Statement (iii) can be derived from \cite[Theorem~4, Section~5.9.2, p.~288]{evans}
as follows:
\begin{align*}
\norm{\partial_t^i D^{j}\bw(0)}{H^1(D)}
&\lesssim
\norm{\partial_t^i D^{j}\bw}{L^2(0,T;H^2(D))}
+
\norm{\partial_t^{i+1} D^{j}\bw}{L^2(0,T;L^2(D))}
\\
&\lesssim
\norm{\partial_t^i \bw}{L^2(0,T;H^{j+2}(D))}
+
\norm{\partial_t^{i+1} \bw}{L^2(0,T;H^{j}(D))}
\\
&\lesssim
\norm{\bw}{H^{k,2k}(D_T)}
\end{align*}
if $i+j/2\le k-1$ and $k\ge1$.
The lemma is proved.
\end{proof}

The following lemma states some useful inequalities involving the
norm and seminorm of~$H^{k,2k}(D_T)$.

\begin{lemma}\label{lem:regnorm}
Let $v$, $w$, $\bv$, and $\bw$ be scalar and vector functions in~$H^{k,2k}(D_T)$
for~$k\geq 2$. 
\begin{enumerate}
\renewcommand{\labelenumi}{(\roman{enumi})}
\item
If~$i, j\in\N_0$ satisfy~$0<m=\lceil{i+j/2}\rceil\le k$ 
then~$\partial_t^i D^{j}v\in H^{k-m,2k-2m}(D_T)$ and
\begin{equation}\label{eq:Laplace2}
\norm{\partial_t^i D^{j}v}{H^{k-m,2k-2m}(D_T)}
\leq C |v|_{H^{k,2k}(D_T)}.
\end{equation}
\item
Furthermore,~$vw$, $v\bw$, $\bv\times\bw$, $\bv\cdot\bw$,
and~$|\bv|^2-|\bw|^2$ belong to the corresponding space~$H^{k,2k}(D_T)$
and satisfy
\begin{subequations}\label{eq:prodnorms}
\begin{align}
\norm{vw}{H^{k,2k}(D_T)}&\leq C \norm{v}{H^{k,2k}(D_T)}\norm{w}{H^{k,2k}(D_T)},\label{eq:prod0}\\
\norm{v\bw}{H^{k,2k}(D_T)}&\leq C \norm{v}{H^{k,2k}(D_T)}\norm{\bw}{H^{k,2k}(D_T)},\label{eq:prod3}\\
 \norm{\bv\times\bw}{H^{k,2k}(D_T)}&\leq C \norm{\bv}{H^{k,2k}(D_T)}\norm{\bw}{H^{k,2k}(D_T)},\label{eq:prod1}\\
\norm{\bv\cdot\bw}{H^{k,2k}(D_T)}&\leq C \norm{\bv}{H^{k,2k}(D_T)}\norm{\bw}{H^{k,2k}(D_T)},\label{eq:prod2}\\
\norm{|\bv|^2-|\bw|^2}{H^{k,2k}(D_T)}&\leq C
(\norm{\bv}{H^{k,2k}(D_T)}+\norm{\bw}{H^{k,2k}(D_T)})\norm{\bv-\bw}{H^{k,2k}(D_T)}.\label{eq:weird}
\end{align}
\end{subequations}
\end{enumerate}
The constant $C>0$ depends only on an upper bound of $k$ and on $D_T$.
\end{lemma}
\begin{proof}
To see~\cref{eq:Laplace2}, we use the definition of the $H^{k,2k}(D_T)$-norm
and write
\begin{align*}
 \norm{\partial_t^iD^jv}{H^{k-m,2k-2m}(D_T)}&=
\sum_{\ell=0}^{k-m}\norm{\partial_t^{\ell+i}D^jv}{L^2(0,T;H^{2k-2m-2\ell}(D))}
\\
&=\sum_{\ell=i}^{k-m+i}
\norm{\partial_t^{\ell}D^jv}{L^2(0,T;H^{2k-2m-2\ell+2i}(D))}.
\end{align*}
Since $m=\lceil{i+j/2}\rceil$, we have 
\begin{equation}\label{eq:kml}
k-m+i\le k \quad\text{and}\quad 2k-2m+2i+j\le 2k.
\end{equation}
Hence, if~$i>0$ then
\begin{align*}
\norm{\partial_t^iD^jv}{H^{k-m,2k-2m}(D_T)}
&\le
\sum_{\ell=1}^{k}
\norm{\partial_t^{\ell}v}{L^2(0,T;H^{2k-2m-2\ell+2i+j}(D))}
\\
&\le
\sum_{\ell=1}^{k}
\norm{\partial_t^{\ell}v}{L^2(0,T;H^{2k-2\ell}(D))}
\le
|v|_{H^{k,2k}(D_T)}.
\end{align*}
If $i=0$ then $1 \le j \le 2k$ (as $0<m\le k$) and thus
\begin{align*}
\norm{\partial_t^i&D^jv}{H^{k-m,2k-2m}(D_T)}\\
&=
\sum_{\ell=0}^{k-m}
\norm{\partial_t^{\ell}D^jv}{L^2(0,T;H^{2k-2m-2\ell}(D))}
\\
&\le
\norm{D^jv}{L^2(0,T;H^{2k-2m}(D))}
+
\sum_{\ell=1}^{k}
\norm{\partial_t^{\ell}v}{L^2(0,T;H^{2k-2m-2\ell+j}(D))}
\\
&=
\Big(
\sum_{j'=0}^{2k-2m}
\norm{D^{j+j'}v}{L^2(D_T)}^2
\Big)^{1/2}
+
\sum_{\ell=1}^{k}
\norm{\partial_t^{\ell}v}{L^2(0,T;H^{2k-2m-2\ell+j}(D))}
\\
&=
\Big(
\sum_{j'=j}^{2k-2m+j}
\norm{D^{j'}v}{L^2(D_T)}^2
\Big)^{1/2}
+
\sum_{\ell=1}^{k}
\norm{\partial_t^{\ell}v}{L^2(0,T;H^{2k-2m-2\ell+j}(D))}
\\
&\le
|v|_{H^{k,2k}(D_T)},
\end{align*}
where in the last step we used~\cref{eq:kml} and the definition of the seminorm.

We next show~\cref{eq:prod0}. The product rule implies
 \begin{align*}
  \norm{vw}{H^{k,2k}(D_T)}&=\sum_{\ell=0}^k\norm{\partial_t^\ell(vw)}{L^2(0,T;H^{2k-2\ell}(D))}\\
  &\lesssim
\sum_{\ell=0}^k
\sum_{j_1+j_2=\ell}
\Big(
\int_0^T
\norm{(\partial_t^{j_1}v)(\partial_t^{j_2}w)}{H^{2k-2\ell}(D))}^2 \,dt
\Big)^{1/2}
\\
 &\lesssim 
\sum_{\ell=0}^k
\sum_{j_1+j_2=\ell}
\sum_{n=0}^{2k-2\ell}
\sum_{i_1+i_2=n}
\Big(
\int_0^T\int_D
|D^{i_1}\partial_t^{j_1}v|^2|D^{i_2}\partial_t^{j_2}w|^2 \,dx\,dt
\Big)^{1/2}.
 \end{align*}
Note that
\[
\frac{i_1}{2}+j_1+\frac{i_2}{2}+j_2 
=
\frac{n}{2}+\ell
\le
k-\ell+\ell = k.
\]
Hence, putting~$\II:= \set{(i_1,i_2,j_1,j_2)\in\N_0}{i_1/2+j_1 + i_2/2 + j_2 \leq k}$
we obtain
\[
\norm{vw}{H^{k,2k}(D_T)}
\lesssim
\sum_{(i_1,i_2,j_1,j_2)\in\II}
\Big(
\int_0^T\int_D
|D^{i_1}\partial_t^{j_1}v|^2|D^{i_2}\partial_t^{j_2}w|^2 \,dx\,dt
\Big)^{1/2}
\leq
S_1 + S_2 + S_3,
\]
where
\[
S_{\nu} :=
\sum_{(i_1,i_2,j_1,j_2)\in\II_\nu}
\Big(
\int_0^T\int_D
|D^{i_1}\partial_t^{j_1}v|^2|D^{i_2}\partial_t^{j_2}w|^2 \,dx\,dt
\Big)^{1/2},
\quad \nu=1,2,3,
\]
with
\begin{align*}
\II_1 &:=
\set{(i_1,i_2,j_1,j_2)\in\II}{i_1/2+j_1 \geq 1 \text{ \ and \ } i_2/2+j_2\geq 1},
\\
\II_2 &:=
\set{(i_1,i_2,j_1,j_2)\in\II}{i_1/2+j_1 = 0 \text{ \ or \ } i_2/2+j_2 = 0},
\\
\II_3 &:=
\set{(i_1,i_2,j_1,j_2)\in\II}{i_1/2+j_1 = 1/2 \text{ \ or \ } i_2/2+j_2= 1/2}.
\end{align*}
Each term in~$S_1$ is estimated by using the H\"older inequality separately in
space and time as
\begin{align*}
S_1
&\le
\sum_{(i_1,i_2,j_1,j_2)\in\II_1}
\Big(
\int_0^T
\norm{D^{i_1}\partial_t^{j_1}v(t)}{L^\infty(D)}^2
\norm{D^{i_2}\partial_t^{{j_2}}w(t)}{L^2(D)}^2\,dt
\Big)^{1/2}
\\
&\le
\sum_{(i_1,i_2,j_1,j_2)\in\II_1}
\norm{D^{i_1}\partial_t^{j_1}v}{L^2(0,T;L^\infty(D))}
\norm{D^{i_2}\partial_t^{j_2}w}{L^\infty(0,T;L^2(D))}
\\
&\le
\sum_{(i_1,i_2,j_1,j_2)\in\II_1}
\norm{D^{i_1}\partial_t^{j_1}v}{H^{1,2}(D_T)}
\norm{D^{i_2}\partial_t^{j_2}w}{H^{1,2}(D_T)},
\end{align*}
where in the last step we used~\cref{lem:sobolev}~(i). 
Note that in this index set~$\II_1$ there hold~$\lceil{i_1/2+j_1}\rceil\le k-1$
and~$\lceil{i_2/2+j_2}\rceil\le k-1$. Hence, estimate~\cref{eq:Laplace2}
gives 
\[
\norm{D^{i_1}\partial_t^{j_1}v}{H^{1,2}(D_T)}
\lesssim 
\norm{v}{H^{k,2k}(D_T)}
\quad\text{and}\quad 
\norm{D^{i_2}\partial_t^{j_2}w}{H^{1,2}(D_T)}
\lesssim 
\norm{w}{H^{k,2k}(D_T)},
\] 
implying~$S_1 \lesssim \norm{v}{H^{k,2k}(D_T)} \norm{w}{H^{k,2k}(D_T)}$.

The sum~$S_2$ is estimated with the help of \cref{lem:sobolev}~(ii) by
\begin{align*}
S_2 
&\le
\sum_{\ell=0}^k
\sum_{i=0}^{2k-2\ell}
\Big(
\int_0^T\int_D
|v|^2|D^{i}\partial_t^{\ell}w|^2 \,dx\,dt
\Big)^{1/2}
\\
&\qquad +
\sum_{\ell=0}^k
\sum_{i=0}^{2k-2\ell}
\Big(
\int_0^T\int_D
|D^{i}\partial_t^{\ell}v|^2 
|w|^2
\,dx\,dt
\Big)^{1/2}
\\
&\lesssim
\norm{v}{L^\infty(D_T)}
\sum_{\ell=0}^k
\sum_{i=0}^{2k-2\ell}
\norm{\partial_t^\ell w}{L^2(0,T;H^i(D))}\\
&\qquad 
+
\norm{w}{L^\infty(D_T)}
\sum_{\ell=0}^k
\sum_{i=0}^{2k-2\ell}
\norm{\partial_t^\ell v}{L^2(0,T;H^i(D))}
\\
&\lesssim
\norm{v}{H^{k,2k}(D_T)}
\norm{w}{H^{k,2k}(D_T)}.
\end{align*}

Finally, for $S_3$, since the problem is symmetric, we just consider the case 
when $i_1=1$ and $j_1=0$.
Since~$H^1(D)\subseteq L^6(D)\subseteq L^4(D)$
(see \cref{lem:sobolev}~(i)) we have
\begin{align*}
 \Big(\int_0^T\int_D|D^1v|^2|D^{i_2}\partial_t^{j_2}w|^2\,dx\,dt\Big)^{1/2}&\leq
\Big(\int_0^T \norm{D^1v(t)}{L^4(D)}^2\norm{D^{i_2}\partial_t^{j_2}w(t)}{L^4(D)}^2\,dt\Big)^{1/2}\\
 &\leq \norm{D^1v}{L^\infty(0,T;L^4(D))}\norm{D^{i_2}\partial_t^{j_2}w}{L^2(0,T;L^4(D))}\\
 &\lesssim \norm{D^1v}{H^1(0,T;H^1(D))}\norm{D^{i_2}\partial_t^{j_2}w}{L^2(0,T;H^1(D))}\\
 &\leq \norm{v}{H^{2,4}(D_T)}\norm{w}{H^{k,2k}(D_T)}\\
 &\leq
\norm{v}{H^{k,2k}(D_T)}\norm{w}{H^{k,2k}(D_T)},
\end{align*}
where in the penultimate step we used~\cref{eq:Laplace2}, noting 
that~$i_2/2+j_2<k$.
This and the analogous result for $i_2=1$ and $j_2=0$ prove
\[
S_3 \lesssim
\norm{v}{H^{k,2k}(D_T)}\norm{w}{H^{k,2k}(D_T)}.
\]
Altogether, we obtain~\cref{eq:prod0}.

The remaining multiplicative estimates~\cref{eq:prod3}--\cref{eq:prod2}
follow from~\cref{eq:prod0} by the fact that all of them can be expressed
as (sums of) products of scalar functions.

Finally, we show~\cref{eq:weird} {by using} the identity
{$|\bv|^2-|\bw|^2=(\bv+\bw)\cdot(\bv-\bw)$} and the already proved
estimate~\cref{eq:prod2}. This concludes the proof.
\end{proof}

The following lemma is a slight generalization to the vector case
of a well-known result on the existence of solutions of the heat equation.
\begin{lemma}\label{lem:heatreg}
Let $L\colon\R^3\to\R^3$ denote {a linear} operator which satisfies
\begin{equation}\label{eq:L}
L\ba\cdot\ba\geq c_L|\ba|^2 
\quad\text{for all } \ba\in\R^3,
\end{equation}
for some $c_L>0$.
For {a given} $\br\in L^2(D_T)$, the vector-valued heat equation
\begin{align}\label{eq:heat}
\begin{split}
 L\partial_t\bw -\Delta\bw &=\br\quad\text{in }D_T,\\
 \bw&=0 \quad{\text{in } \{0\} \times D,} \\
 \partial_n\bw&=0\quad\text{on }\Gamma_T
 \end{split}
\end{align}
has a weak solution which satisfies
\begin{align}\label{eq:heatreg}
 \norm{\bw}{H^{1,2}(D_T)}
\leq C_{\rm heat}\norm{\br}{L^2(D_T)}^2.
\end{align}
The constant $C_{\rm heat}>0$ depends only on~$T$, $L$, and $D$.
\end{lemma}
\begin{proof}
{
Note that \cref{eq:L} implies the existence of~$L^{-1}$ which satisfies 
\begin{equation}\label{eq:L inv}
 L^{-1}\bb\cdot \bb \geq \frac{c_L}{\norm{L}{}^2}|\bb|^2
\quad\text{for all } \bb\in\R^3.
\end{equation}
}
Thus, we can reformulate~\cref{eq:heat}
{into}
\begin{align}\label{eq:heat2}
 \begin{split}
 \partial_t\bw -L^{-1}\Delta\bw &=L^{-1}\br\quad\text{in }D_T,\\
 \bw&=0 \quad{\text{in } \{0\} \times D,} \\
 \partial_n\bw&=0\quad\text{on }\Gamma_T.
 \end{split}
\end{align}
We want to use the result~\cite[Theorem~3.2]{lions}.
To that end, and in the notation of~\cite{lions}, we define $A:=-L^{-1}\Delta$ and 
\begin{align*}
 D(A):=\set{\bv\in H^2(D)}{\partial_n\bv =0\text{ on }\Gamma}\subseteq L^2(D).
\end{align*}
Define the graph norm $\norm{\cdot}{D(A)}^2:=\norm{\cdot}{L^2(D)}^2 +
\norm{A(\cdot)}{L^2(D)}^2$. Then, there holds for all $p\in\C$ {satisfying} ${\rm
Re}(p)\geq p_0>0$ and {for} all ${\bv}\in D(A)$
\begin{align*}
 \norm{(A+p)\bv}{L^2(D)}\leq (1+|p|)\norm{\bv}{D(A)}
\end{align*}
as well as
\begin{align*}
 \norm{(A+p)\bv}{L^2(D)}^2 &= \norm{A\bv}{L^2(D)}^2 -2{\rm Re}(p\dual{A \bv}{\bv}_D) + |p|^2\norm{\bv}{L^2(D)}^2\\
 &=  \norm{A\bv}{L^2(D)}^2 
+
2{\rm Re}(p\dual{L^{-1}\nabla \bv}{\nabla \bv}_D) 
+ |p|^2\norm{\bv}{L^2(D)}^2.
\end{align*}
It follows from~\cref{eq:L inv} that
\[
2{\rm Re}(p\dual{L^{-1}\nabla \bv}{\nabla \bv}_D) 
\geq
2p_0 \frac{c_L}{\norm{L}{}^{2}}
\norm{\nabla \bv}{L^2(D)}^2 
\geq 0,
\]
so that
\begin{align*}
 \norm{(A+p)\bv}{L^2(D)}^2 
 &\geq \norm{A\bv}{L^2(D)}^2 + |p|^2\norm{\bv}{L^2(D)}^2\geq \min\{1,p_0^2\}\norm{\bv}{D(A)}^2.
\end{align*}
Standard elliptic regularity theory (see e.g.~\cite[Theorem~4.18]{mclean}) shows that $A+p\colon D(A)\to L^2(D)$ is surjective.
Hence, $A+p\colon D(A)\to L^2(D)$ is a bijective isomorphism. Moreover, we {have} 
for $\bv\in L^2(D)$
\begin{align*}
 \norm{(A+p)^{{-1}}\bv}{L^2(D)}\lesssim \frac{1}{1+|p|}\norm{\bv}{L^2(D)}.
\end{align*}
for all $p\in\C$ satisfying ${\rm Re}(p)>p_0$. Thus, the requirements
of~\cite[Theorem~3.2]{lions} are satisfied {which yields the existence of}
$\bw\in L^2(0,T;D(A))$ satisfying~\cref{eq:heat2} and hence also~\cref{eq:heat}. 

Standard elliptic regularity theory (see e.g.~\cite[Theorem~4.18]{mclean})
gives
\[
\norm{\bw}{H^2(D)}\lesssim
\norm{\Delta\bw}{L^2(D)}+\norm{\bw}{H^1(D)}
\quad\text{for all } \bw\in D(A). 
\]
Since $\norm{\bw}{H^1(D)}\lesssim \norm{\Delta\bw}{L^2(D)} + \norm{\bw}{L^2(D)}$
for all functions satisfying $\partial_n\bw =0$,
we {deduce} that $\bw\in L^2(0,T;H^2(D))$. The proof of~\cite[Theorem~3.2]{lions} also reveals
\begin{align*}
 \norm{\bw}{L^2(0,T;H^2(D))}\lesssim  \norm{\bw}{L^2(0,T;D(A))}\lesssim \norm{L^{-1}\br}{L^2(D_T)}\simeq \norm{\br}{L^2(D_T)}.
\end{align*}
This estimate and~\cref{eq:heat2} yield
$\norm{\partial_t\bw}{L^2(0,T;L^2(D))}\lesssim  \norm{\br}{L^2(D_T)}$, 
completing the proof of the lemma.
\end{proof}

The next lemma is a result on higher regularity for solutions to~\cref{eq:heat}.
\begin{lemma}\label{lem:heatreg2}
Under the assumption of \cref{lem:heatreg}, if
$\br\in H^{k-1,2k-2}(D_T)$ for $k\geq 2$ satisfies 
\[
\partial_t^i D^{j} \br(0) \in H^1_\star(D) \quad\text{for all } i+j/2\leq k-2, 
\]
then
the solution $\bw$ of the vector-valued heat equation~\cref{eq:heat} satisfies
\begin{equation}\label{eq:heatreg2}
 \norm{\bw}{H^{k,2k}(D_T)}\leq C_{\br} \norm{\br}{H^{k-1,2k-2}(D_T)}
\end{equation}
and
\begin{equation}\label{eq:heatreg3}
\partial_t^i  D^{j}  \bw(0)\in H^1_\star(D)
\quad\text{for all } i+j/2\leq k-1. 
\end{equation}
\end{lemma}
\begin{proof}
We first recall that if~$\br\in H^{k-1,2k-2}(D_T)$ then~$\partial_t^i D^{j}
\br(0) \in H^1(D)$ for~$i+j/2\leq k-2$; see \cref{lem:sobolev}.
 The proof is an induction on $k\in\N$, where \cref{lem:heatreg} confirms the case $k=1$.
 Let $k>1$ and assume that~\cref{eq:heatreg2} and~\cref{eq:heatreg3} hold for
$k-1$.
 Then, differentiation reveals that $\bv:=\partial_t\bw-L^{-1}\br(0)$ is the
unique solution of
\begin{align*}
\begin{split}
 L\partial_t\bv -\Delta\bv &=\partial_t\br+L^{-1}\Delta\br(0)\quad\text{in }D_T,\\
 \bv&=0 \quad\text{in } \{0\} \times D,\\
 \partial_n\bv&=0\quad\text{on }\Gamma_T.
 \end{split}
\end{align*}
The right-hand side $\widetilde{\br}:=\partial_t\br+L^{-1}\Delta\br(0)$ satisfies $\partial_t^i D^{j}  \widetilde{\br}(0) \in H^1_\star(D)$ for all $i+j/2\leq k-3$.
The induction hypothesis and \cref{lem:regnorm} show that
\begin{align}\label{eq:reg0}
 \norm{\bv}{H^{k-1,2k-2}(D_T)}\lesssim \norm{\partial_t\br+L^{-1}\Delta\br(0)}{H^{k-2,2k-4}(D_T)}\lesssim \norm{\br}{H^{k-1,2k-2}(D_T)}
\end{align}
as well as
\begin{align}\label{eq:null}
 \partial_t^n  D^{m} \bv(0)\in H^1_\star(D)\quad\text{for all }n+m/2\leq k-2.
\end{align}
The definition of~$\bv$ and estimate~\cref{eq:reg0} imply
\begin{align}\label{eq:reg1}
 \sum_{j=1}^k\norm{\partial_t^j\bw}{L^2(0,T;H^{2k-2j}(D))}
\lesssim
\norm{\bv}{H^{k-1,2k-2}(D_T)}
\lesssim 
\norm{\br}{H^{k-1,2k-2}(D_T)}.
\end{align}

Assume for the moment that $\bw$ and $\br$ are smooth. Then, we have with elliptic regularity (see, e.g.,~\cite[Theorem~4.18]{mclean}) and  $-\Delta\bw = \br-L\partial_t\bw$ that all $0\leq t\leq T$ satisfy
\begin{align*}
 \norm{\bw(t)}{H^{2k}(D)}\lesssim \norm{\br(t)}{H^{2k-2}(D)}+\norm{\partial_t\bw(t)}{H^{2k-2}(D)}.
\end{align*}
Integration over $t\in(0,T)$ reveals for smooth $\bw$ and $\br$
\begin{align}\label{eq:reg2}
 \norm{\bw}{L^2(0,T;H^{2k}(D))}\lesssim \norm{\br}{H^{k-1,2k-2}(D_T)}+\norm{\partial_t\bw}{L^2(0,T;H^{2k-2}(D))}.
\end{align}
A density argument now proves $\bw\in L^2(0,T;H^{2k}(D))$ with~\cref{eq:reg2} even for non-smooth $\bw$.
The combination of~\cref{eq:reg1} and~\cref{eq:reg2} shows
\begin{align*}
 \norm{\bw}{H^{k,2k}(D_T)}\lesssim \norm{\br}{H^{k-1,2k-2}(D_T)}.
\end{align*}
To see $\partial_t^i  D^{j}  \bw(0)\in H^1_\star(D)$ for all $i+j/2\leq k-1$, we distinguish three cases: First, for $i\geq 2$,
since~$\partial_t^i D^{j}\bw=\partial_t^{i-1}D^{j}\bv$,
property~\cref{eq:null} gives with $n=i-1$ and $m=j\leq 2k-2-2i$ that
\begin{align*}
 \partial_t^i D^{j}  \bw(0)=\partial_t^n D^{m}  \bv(0) \in
H^1_\star(D).
\end{align*}
Second, for $i=1$,~\cref{eq:null} shows with $n=0$ and $m=j\leq 2k-4$ that
\begin{align*}
  D^{m} \bv(0)=\partial_t D^{j}  \bw(0) - L^{-1} D^{j} \br(0) \in H^1_\star(D).
\end{align*}
Since $ D^{j} \br(0)\in H^1_\star(D)$ for all $j/2\leq k-2$ by definition, we
obtain $ \partial_t D^{j} \bw(0)\in H^1_\star(D)$ for all $j/2\leq k-2$. 
Finally, for $i=0$, we have for $ D^{j}  \bw(0)=0 \in H^1_\star(D)$
 for $j/2\leq k-1$ by definition. Altogether, this proves $\partial_t^i\Delta^j\bw(0)\in H^1_\star(D)$ for all $i+j/2\leq k-1$ and thus concludes the proof.
\end{proof}

The next technical result will be used to prove that the solution of some
nonlinear parabolic problem satisfies condition~\cref{eq:con2}
for all~$t>0$ if it satisfies that condition at~$t=0$.

\begin{lemma}\label{lem:quasilin}
Let $u\in H^1(0,T;L^2(D))$ such that $u(t)\in W^{2,\infty}(D)$ for all $0\leq t \leq T$  with $u|_{\{0\}\times D}=1$ be a strong solution of
 \begin{align*}
 \beta \partial_t u - u\Delta u &=0\quad\text{in }D_T,\\
 \partial_n u &=0\quad\text{on }\Gamma_T
 \end{align*}
 for some constant $\beta>0$.
Then, there holds $u=1$ in $D_T$.
\end{lemma}
\begin{proof}
Define $e:=u-1$. There holds
\begin{align*}
 \beta\partial_t e - e\Delta e - \Delta e=0\quad\text{and}\quad \partial_n e=0\text{ on }\Gamma_T.
\end{align*}
Multiplication by $e$ and integration by parts over $D$ shows
\begin{align*}
\frac{\beta}{2}\partial_t \norm{e(t)}{L^2(D)}^2 +\norm{\nabla e}{L^2(D)}^2\leq \norm{\Delta e(t)}{L^\infty(D)}\norm{e(t)}{L^2(D)}^2\lesssim \norm{e(t)}{L^2(D)}^2,
\end{align*}
by use of the regularity assumptions for the last inequality.
Thus, we have
\begin{align*}
\partial_t \norm{e(t)}{L^2(D)}^2\lesssim  \norm{e(t)}{L^2(D)}^2\quad\text{for all }0\leq t\leq T.
\end{align*}
Gronwall's inequality proves $\norm{e(t)}{L^2(D)}\lesssim \norm{e(0)}{L^2(D)}=0$, which concludes the proof.
\end{proof}

We next define a residual operator which will be used to generate a
sequence~$\{\bm_{\ell}\}$ converging to a solution~$\bm$
of~\cref{eq:llg}--\cref{eq:con}.
\begin{definition}\label{def:residual}
Let $x_0$ be an arbitrary point in~$D$ and~$\bm^0$ be the initial data given
in~\cref{eq:con3}. For any~$\bv\in H^{k,2k}(D_T)$ for some $k>0$, we define the
residual
\begin{equation}\label{eq:orig2}
\RR(\bv)
:=
\alpha\bv_t+\bv\times \bv_t- C_e|\bv|^2\Delta\bv -C_e|\nabla\bv|^2\bv.
\end{equation}
We also define a linear operator $L\colon\R^3\to\R^3$ by 
\begin{equation}\label{eq:L def}
L\ba := L_{\bm^0(x_0)}\ba
:=
\alpha\ba +\bm^0(x_0)\times \ba,
\quad \ba\in\R^3.
\end{equation}
\end{definition}

It is easy to see that $L$ satisfies~\cref{eq:L} with $c_L=\alpha$ and that
 \begin{align}
 \RR(\bv)
  &=\alpha\bv_t+\bv\times \bv_t-C_e\Delta\bv+ C_e(1-|\bv|^2)\Delta\bv
-C_e|\nabla\bv|^2\bv \notag\\
  &= \alpha\bv_t+\bm^0(x_0)\times\bv_t +(\bv-\bm^0(x_0))\times \bv_t-C_e\Delta\bv 
\notag \\
  &\quad
+ C_e(1-|\bv|^2)\Delta\bv -C_e|\nabla\bv|^2\bv \notag\\
  &= L\bv_t+(\bv-\bm^0(x_0))\times \bv_t-C_e\Delta\bv 
  + C_e(1-|\bv|^2)\Delta\bv -C_e|\nabla\bv|^2\bv,
\label{eq:form3}
 \end{align}
where $L$ is applied pointwise in time and space.

The following lemma gives some mapping properties of the operator~$\RR$.
(We recall the definition of~$H^1_{\star}(D)$ in~\cref{eq:H1 sta}.)
\begin{lemma}\label{lem:residual}
\mbox{}
\begin{enumerate}
\renewcommand{\labelenumi}{(\roman{enumi})}
\item
The residual operator $\RR$ defined in \cref{def:residual} is continuous 
from~$H^{k,2k}(D_T)$ into $H^{k-1,2k-2}(D_T)$ for $k\geq3$.
More precisely, there holds
\begin{align*}
\norm{\RR(\bv)&-\RR(\bw)}{H^{k-1,2k-2}(D_T)}\\
&\le
C_{\RR}(1+\norm{\bv}{H^{k,2k}(D_T)}^2+\norm{\bw}{H^{k,2k}(D_T)}^2)
\norm{\bv-\bw}{H^{k,2k}(D_T)}.
\end{align*}
\item
For~$k\ge 3$, if $\bw\in H^{k,2k}(D_T)$ satisfies
\[
\partial_t^iD^{j}\bw(0)\in H_{\star}^1(D) 
\quad\text{for all } i+j/2\le k-1,
\]
then
\[
\partial_t^i D^{j} \RR(\bw)(0)\in H^1_\star(D) 
\quad\text{for all } i+j/2\leq k-2.
\]
\end{enumerate}
\end{lemma}
\begin{proof}
Statement (i) is proved by using \cref{lem:regnorm} 
(which is applicable because $k\ge3$) as follows:
\begin{align*}
 &\norm{\RR(\bv)-\RR(\bw)}{H^{k-1,2k-2}(D_T)}\nonumber\\
 &\lesssim \norm{\partial_t(\bv-\bw)}{H^{k-1,2k-2}(D_T)}\nonumber\\
 &\qquad +
\norm{(\bv-\bw)\times\partial_t\bv}{H^{k-1,2k-2}(D_T)}+\norm{\bw\times\partial_t(\bv-\bw)}{H^{k-1,2k-2}(D_T)}\nonumber\\
   &\qquad
+\norm{(|\bv|^2-|\bw|^2)\Delta\bv}{H^{k-1,2k-2}(D_T)}+\norm{|\bw|^2\Delta(\bv-\bw)}{H^{k-1,2k-2}(D_T)}\nonumber\\
 &\qquad +\norm{(\nabla(\bv-\bw)\cdot\nabla\bv)\bv}{H^{k-1,2k-2}(D_T)}
+\norm{(\nabla\bw\cdot\nabla(\bv-\bw))\bv}{H^{k-1,2k-2}(D_T)}\nonumber\\
 &\qquad+\norm{(\nabla\bw\cdot\nabla\bw)(\bv-\bw)}{H^{k-1,2k-2}(D_T)}\\
 &\lesssim \big(1+\norm{\bv}{H^{k,2k}(D_T)} + \norm{\bw}{H^{k,2k}(D_T)}\\
 &\qquad\quad
+\norm{\bv}{H^{k,2k}(D_T)}^2+ \norm{\bw}{H^{k,2k}(D_T)}^2\big)\nonumber
 \norm{\bv-\bw}{H^{k,2k}(D_T)}\\
 &\lesssim (1+\norm{\bv}{H^{k,2k}(D_T)}^2+\norm{\bw}{H^{k,2k}(D_T)}^2) \norm{\bv-\bw}{H^{k,2k}(D_T)}.
\end{align*}
To prove (ii) we note that since $\RR(\bw)\in H^{k-1,2k-2}(D_T)$
for~$k\ge 3$, \cref{lem:sobolev} gives
$\partial_t^i D^{j} \RR(\bw)(0)\in H^1(D)$ for all $i+j/2\leq k-2$.
It remains to show that 
$\Delta \partial_t^iD^{j}\RR(\bw)(0)\in L^2(D)$ 
and that the normal derivative of~$\partial_t^iD^{j}\RR(\bw)(0)$ is zero.
It is easy to see from the definition~\cref{eq:orig2} of $\RR$ and
the product rule that for $i+j/2\leq k-2$, the derivative
$\partial_t^iD^{j}\RR(\bw)$ is a sum of terms of the form
\begin{align}\label{eq:terms}
 (\partial_t^{n_1} D^{m_1}\bv_1)\odot_1 
\Big(
(\partial_t^{n_2} D^{m_2}\bv_2)\odot_2 (\partial_t^{n_3} D^{m_3}\bv_3)
\Big)
\end{align}
with $n_1+n_2+n_3+(m_1+m_2+m_3)/2\leq k-1$ and $\bv_s\in\{\bw,1\}$, $s=1,2,3$,
where $\odot_1$ and $\odot_2$ denote either the scalar, dot, or crossproduct.
Thus at least 2 elements in the set
$\{(n_1,m_2),(n_2,m_2),(n_3,m_3)\}$ satisfy
$n_i+m_i/2\leq (k-1)/2$. Without loss of generality we assume 
$i=2,3$. \cref{lem:regnorm} gives
\[
\partial_t^{n_i} D^{m_i} \bw\in
H^{k-\lceil(k-1)/2\rceil,2(k-\lceil(k-1)/2\rceil)}(D_T)\subseteq H^{2,4}(D_T)
\]
(because~$k\ge3$).
\cref{lem:sobolev}~(iii)\&(i) imply
\[
D^2(\partial_t^{n_i} D^{{m_i}}\bw(0))= \partial_t^{n_i} D^{m_i+2}\bw(0)\in
H^1(D)\subset L^6(D)\subset L^4(D),
\]
and thus
\[
\partial_t^{n_i} D^{m_i}\bw(0),\partial_t^{n_i} 
D^{m_i+1}\bw(0)\in H^2(D)\subseteq L^\infty(D), \quad i=1,2.
\]
The product rule shows (with the definition $\Delta^{1/2}:=D^1$) that $\Delta\partial_t^iD^{j}\RR(\bw)(0)$ is a sum of terms of the form
\begin{align}\label{eq:terms2}
 (\partial_t^{n_1}\Delta^{r_1} D^{m_1}\bv_1(0))\odot_1\Big( (\partial_t^{n_2}\Delta^{r_2} D^{m_2}\bv_2(0))\odot_2 (\partial_t^{n_3}\Delta^{r_3} D^{m_3}\bv_3(0))\Big)
\end{align}
with 
$r_s\in\{0,1/2,1\}$, $s=1,2,3$, satisfying $r_1+r_2+r_3=1$.
This and the considerations above 
together with the assumption $\partial_t^n D^{m}\bv_1(0)\in H^1_\star(D)$ show
\begin{align*}
 \big(\partial_t^{n_1}&\Delta^{r_1} D^{m_1}\bv_1(0), \partial_t^{n_2}\Delta^{r_2} D^{m_2}\bv_2(0),\partial_t^{n_3}\Delta^{r_3} D^{m_3}\bv_3(0)\big)\\
 &\in 
\begin{cases}
       L^2(D)\times L^\infty(D)\times L^\infty(D) &\text{for }r_1\in\{1/2,1\},
\ r_2,r_3\in \{0,1/2\}, \\
       H^1(D)\times L^4(D)\times L^\infty(D) &\text{for }r_1=0, r_2=1, r_3=0,\\
       H^1(D)\times L^\infty(D)\times L^4(D) &\text{for }r_1=0, r_2=0, r_3=1,\\
       H^1(D)\times L^\infty(D)\times L^\infty(D) &\text{for }r_1=0, r_2=r_3=1/2,\\
\end{cases}
\\
&\subseteq
\begin{cases}
       L^2(D)\times L^\infty(D)\times L^\infty(D) &\text{for }r_1\in\{1/2,1\},
\ r_2,r_3\in \{0,1/2\}, \\
       L^4(D)\times L^4(D)\times L^\infty(D) &\text{for }r_1=0, r_2=1, r_3=0,\\
       L^4(D)\times L^\infty(D)\times L^4(D) &\text{for }r_1=0, r_2=0, r_3=1,\\
       L^2(D)\times L^\infty(D)\times L^\infty(D) &\text{for }r_1=0, r_2=r_3=1/2.\\
\end{cases}
\end{align*}
Hence the product~\cref{eq:terms2} is in $L^2(D)$. This implies that
$\Delta\partial_t^iD^{j}\RR(\bw)(0)\in L^2(D)$.
Moreover, the normal derivatives of each factor of~\cref{eq:terms} are zero by definition, and thus the product rule
implies that also $\partial_t^i D^{j}\RR(\bw)(0)=0$, completing the proof of the lemma.
\end{proof}

The following lemma gives sufficient conditions for a given function~$\bm$
satisfying~$\RR(\bm)=0$ to be a solution to~\cref{eq:llg}--\cref{eq:con}.
\begin{lemma}\label{lem:dn}
If $\bm\in H^{k,2k}(D_T)$ for $k\geq 3$ satisfies 
\begin{equation}\label{eq:Rm}
\begin{aligned}
\RR(\bm)&=0 \text{ on } D_T,
\\
\partial_n\bm&=0 \text{ on } \Gamma_T,
\\
|\bm|&=1 \text{ on } \{0\}\times D,
\\
\bm(0,\cdot) &= \bm^0 \quad\text{in } D,
\end{aligned}
\end{equation}
then $\bm$ is a strong solution to~\cref{eq:llg}--\cref{eq:con}.

\end{lemma}
\begin{proof}
It suffices to show that~$\bm$
satisfies~\cref{eq:con2} and~\cref{eq:llg2}.
The first property is shown by invoking \cref{lem:quasilin}. To this end, 
let~$u:=|\bm|^2$.
\cref{lem:sobolev}~(ii) shows $\bm\in W^{1,\infty}(D_T)$ and hence $u\in W^{1,\infty}(D_T)$. \cref{lem:sobolev}~(iii) proves
$\bm(t)\in H^5(D)\subseteq W^{2,\infty}(D)$, which implies $u(t)\in W^{2,\infty}(D)$ for all $0\leq t\leq T$.
Moreover, by using
\begin{equation}\label{eq:Dm}
\Delta|\bm|^2 = 2\Delta\bm\cdot\bm + 2|\nabla\bm|^2,
\end{equation}
together with~\cref{eq:orig2} and~\cref{eq:Rm} we obtain
\begin{align*}
\frac{\alpha}{2} \partial_t u - \frac{C_e}{2} u \Delta u 
&= 
\frac{\alpha}{2} \partial_t|\bm|^2 - \frac{C_e}{2}|\bm|^2 \Delta|\bm|^2 
\\
&=
\alpha\bm_t\cdot \bm- C_e|\bm|^2\Delta\bm\cdot\bm -C_e|\nabla\bm|^2|\bm|^2\\
&= 
\RR(\bm)\cdot\bm
= 0 \quad\text{in } D_T.
 \end{align*}
Assumption~\cref{eq:Rm} also implies 
$\partial_n u = \partial_n|\bm|^2 = 2\partial_n\bm \cdot\bm =0$ on $\Gamma_T$. Hence,
 \cref{lem:quasilin} yields $u=1$ in $D_T$, i.e.~\cref{eq:con2} holds, 
which in turn together with $\RR(\bm)=0$ implies
\begin{align*}
\alpha\bm_t+\bm\times \bm_t
=
 C_e\Delta\bm +C_e|\nabla\bm|^2\bm.
\end{align*}
It follows from~\cref{eq:Dm} that~$|\nabla\bm|^2=-\Delta\bm\cdot\bm$
so that~$\bm$ satisfies~\cref{eq:llg2}, completing the
proof of the lemma.

\end{proof}

Finally, since $\RR$ is not linear, we need the following lemma to
estimate~$\RR(\bv-\bw)$.
\begin{lemma}\label{lem:huge}
 Let $\bv,\bw\in H^{k,2k}(D_T)$ for $k\geq 3$. Then, there holds
 \begin{align}\label{eq:huge}
\norm{\RR(\bv-\bw)}{H^{k-1,2k-2}(D_T)}
&\lesssim
\norm{\RR(\bv)-(L\partial_t-C_e\Delta)\bw}{H^{k-1,2k-2}(D_T)}
\notag\\
&\quad
+
\norm{\bv-\bm^0(x_0)}{H^{k-1,2k-2}(D_T)}
\norm{\bw}{H^{k,2k}(D_T)}
\notag\\
&\quad
+
\norm{(1-|\bv|^2)\Delta\bw}{H^{k-1,2k-2}(D_T)}^2
\notag\\
&\quad
+
\norm{\bw}{H^{k,2k}(D_T)}
|\bv|_{H^{k,2k}(D_T)}
\big(
1
%+
%|\bv|_{H^{k,2k}(D_T)}
+
\norm{\bv}{H^{k,2k}(D_T)}
\big)
\notag\\
&\quad
+
\norm{\bw}{H^{k,2k}(D_T)}^2
\big(
1
%+
%|\bv|_{H^{k,2k}(D_T)}
+
\norm{\bv}{H^{k,2k}(D_T)}
\big)
%\notag\\
%&\quad
+
\norm{\bw}{H^{k,2k}(D_T)}^3.
\end{align}
The hidden constant depends only on $C_e$ and on the constants from
\cref{lem:regnorm}.
\end{lemma}

\begin{proof}
It can be easily derived from~\cref{eq:form3} that
\begin{align*}
\RR(\bv-\bw) - \RR(\bv)
&=
-
(L\partial_t-C_e\Delta)\bw
+
(\bv-\bm^0(x_0))\times\bw_t
+
\bw\times(\bv_t-\bw_t)
\\
&\quad
-
C_e
\big(
1-|\bv-\bw|^2
\big)
\Delta\bw
-
C_e
\big(
|\bw|^2 - 2(\bv\cdot\bw)
\big)
\Delta\bv
\\
&\quad
+
C_e|\nabla\bv-\nabla\bw|^2 \bw
-
C_e
\big(
|\nabla\bw|^2 - 2\nabla\bv\cdot\nabla\bw
\big)
\bv,
\end{align*}
so that
\begin{align*}
\RR(\bv-\bw) 
&=
\RR(\bv)
-
(L\partial_t-C_e\Delta)\bw
+
(\bv-\bm^0(x_0))\times\bw_t
+
\bw\times(\bv_t-\bw_t)
\\
&\quad
-
C_e
\big(
1-|\bv|^2+2\bv\cdot\bw-|\bw|^2
\big)
\Delta\bw
-
C_e
\big(
|\bw|^2 - 2(\bv\cdot\bw)
\big)
\Delta\bv
\\
&\quad
+
C_e
\big(
|\nabla\bv|^2
-2\nabla\bv\cdot\nabla\bw
+|\nabla\bw|^2 
\big)
\bw
-
C_e
\big(
|\nabla\bw|^2 - 2\nabla\bv\cdot\nabla\bw
\big)
\bv
\\
&=
\RR(\bv)
-
(L\partial_t-C_e\Delta)\bw
+ \bT_1 + \cdots + \bT_6.
\end{align*}
Hence
\begin{align*}
\norm{\RR(\bv&-\bw)}{H^{k-1,2k-2}(D_T)}\\
&\lesssim
\norm{\RR(\bv)-(L\partial_t-C_e\Delta)\bw}{H^{k-1,2k-2}(D_T)}
+
\sum_{i=1}^6
\norm{\bT_i}{H^{k-1,2k-2}(D_T)}.
\end{align*}
Denoting~$T_i = \norm{\bT_i}{H^{k-1,2k-2}(D_T)}$, \cref{lem:regnorm} yields
\begin{align*}
T_1
&\lesssim
\norm{\bv-\bm^0(x_0)}{H^{k-1,2k-2}(D_T)}
|\bw|_{H^{k,2k}(D_T)}
\\
&\le
\norm{\bv-\bm^0(x_0)}{H^{k-1,2k-2}(D_T)}
\norm{\bw}{H^{k,2k}(D_T)},
\\
T_2
&\lesssim
\norm{\bw}{H^{k-1,2k-2}(D_T)}
\big(
|\bv|_{H^{k,2k}(D_T)}
+
|\bw|_{H^{k,2k}(D_T)}
\big)
\\
&\le
|\bv|_{H^{k,2k}(D_T)}
\norm{\bw}{H^{k,2k}(D_T)}
+
\norm{\bw}{H^{k,2k}(D_T)}^2,
\\
T_3
&\lesssim
\norm{(1-|\bv|^2)\Delta\bw}{H^{k-1,2k-2}(D_T)}
+
\norm{\bv}{H^{k-1,2k-2}(D_T)}
\norm{\bw}{H^{k-1,2k-2}(D_T)}
|\bw|_{H^{k,2k}(D_T)}
\\
&\quad
+
\norm{\bw}{H^{k-1,2k-2}(D_T)}^2
|\bw|_{H^{k,2k}(D_T)}
\\
&\le
\norm{(1-|\bv|^2)\Delta\bw}{H^{k-1,2k-2}(D_T)}
+
\norm{\bv}{H^{k,2k}(D_T)}
\norm{\bw}{H^{k,2k}(D_T)}^2
+
\norm{\bw}{H^{k,2k}(D_T)}^3,
\\
T_4
&\lesssim
\norm{\bw}{H^{k-1,2k-2}(D_T)}^2
|\bv|_{H^{k,2k}(D_T)}
+
\norm{\bv}{H^{k-1,2k-2}(D_T)}
\norm{\bw}{H^{k-1,2k-2}(D_T)}
|\bv|_{H^{k,2k}(D_T)}
\\
&\le
\norm{\bv}{H^{k,2k}(D_T)}
\norm{\bw}{H^{k,2k}(D_T)}^2
+
|\bv|_{H^{k,2k}(D_T)}
\norm{\bv}{H^{k,2k}(D_T)}
\norm{\bw}{H^{k,2k}(D_T)}
\\
T_5
&\lesssim
|\bv|_{H^{k,2k}(D_T)}^2
\norm{\bw}{H^{k-1,2k-2}(D_T)}
+
|\bv|_{H^{k,2k}(D_T)}
|\bw|_{H^{k,2k}(D_T)}
\norm{\bw}{H^{k-1,2k-2}(D_T)}
\\
&\quad
+
|\bw|_{H^{k,2k}(D_T)}^2
\norm{\bw}{H^{k-1,2k-2}(D_T)}
\\
&\le
|\bv|_{H^{k,2k}(D_T)}^2
\norm{\bw}{H^{k,2k}(D_T)}
+
\norm{\bv}{H^{k,2k}(D_T)}
\norm{\bw}{H^{k,2k}(D_T)}^2
+
\norm{\bw}{H^{k,2k}(D_T)}^3,
\\
T_6
&\lesssim
|\bw|_{H^{k,2k}(D_T)}^2
\norm{\bv}{H^{k-1,2k-2}(D_T)}
+
|\bv|_{H^{k,2k}(D_T)}
|\bw|_{H^{k,2k}(D_T)}
\norm{\bv}{H^{k-1,2k-2}(D_T)}
\\
&\le
\norm{\bv}{H^{k,2k}(D_T)}
\norm{\bw}{H^{k,2k}(D_T)}^2
+
|\bv|_{H^{k,2k}(D_T)}
\norm{\bv}{H^{k,2k}(D_T)}
\norm{\bw}{H^{k,2k}(D_T)}.
\end{align*}
Collecting all the terms we obtain the desired estimates, completing the proof.
\end{proof}

\section{Proof of the Main Result}
This is a constructive proof. Starting with the initial guess
$\bm_0(t,x):=\bm^0(x)$ for all $(t,x)\in D_T$, we define a 
sequence~$(\bm_\ell)_{\ell\in\N_0}$ as follows.
Having defined~$\bm_\ell$, $\ell=0,1,2,\ldots,$ the construction involves 
the following tasks:
\begin{itemize}
\item
Define $\br_\ell:= \RR(\bm_\ell)$,
\item
Solve
 \begin{alignat*}{2}
  L\partial_t\bR_\ell-C_e \Delta\bR_\ell&=\br_\ell \quad&&\text{in }D_T,\\
  \partial_n\bR_\ell&=0\quad&&\text{on }\Gamma_T,\\
  \bR_\ell&=0\quad&&\text{on }\{0\}\times D.
 \end{alignat*}
\item
Define $\bm_{\ell+1}:= \bm_\ell - \bR_\ell$.
\end{itemize}

First we note that the above iteration is well-defined. Indeed, the
assumptions on the initial data~$\bm^0$ imply that the initial 
guess~$\bm_0$ belongs to $H^{k,2k}(D_T)$ and
satisfies~$\partial_t^iD^{j}\bm_0(0)\in H_\star^1(D)$ for all~$i+j/2\le k-1$.
\Cref{lem:residual,lem:heatreg2} then imply that~$\bR_0$ also has
the same smoothness properties, and so does~$\bm_1$. By repeating the same
argument, all functions~$\bm_\ell$ have the same smoothness properties
as~$\bm_0$, and the sequence~$\{\bm_\ell\}$ is well constructed.
Next we note that, $\bm_\ell|_{\{0\}\times D}=\bm_0|_{\{0\}\times D}=\bm^0$ and
$\partial_n\bm_\ell =\partial_n\bm_0=0$ for all $\ell\in\N$. Note also that due
to \cref{lem:heatreg2}
\begin{equation}\label{eq:rl Rl}
\norm{\bR_\ell}{H^{k,2k}(D_T)}
\le
C_{\br}
\norm{\br_\ell}{H^{k-1,2k-2}(D_T)},
\quad \ell=0,1,2,\ldots.
\end{equation}

We will show that the sequence~$(\bm_\ell)_{\ell\in\N_0}$ converges to a
function~$\bm$. \cref{lem:residual} then yields the convergence of~$\RR(\bm_\ell)$ 
to~$\RR(\bm)$ as~$\ell\to\infty$. \cref{lem:dn} will then be used to
conclude that~$\bm$ is a strong solution of~\cref{eq:llg}--\cref{eq:con}.

To show that~$\{\bm_\ell\}$ is a Cauchy sequence we note that
for~$0\le\ell'\le\ell$
\begin{align}\label{eq:ml mlp}
\norm{\bm_\ell-\bm_{\ell'}}{H^{k,2k}(D_T)}
&\le
\sum_{j=\ell'-1}^{\ell-2}
\norm{\bR_{j+1}}{H^{k,2k}(D_T)}.
%\notag\\
%&\lesssim
%\sum_{j=\ell'}^{\ell-1}
%\norm{\br_j}{H^{k-1,2k-2}(D_T)}
%\notag\\
%&=
%\sum_{j=\ell'-1}^{\ell-2}
%\norm{\RR(\bm_{j+1})}{H^{k-1,2k-2}(D_T)}
%\notag\\
%&=
%\sum_{j=\ell'-1}^{\ell-2}
%\norm{\RR(\bm_{j}-\bR_{j})}{H^{k-1,2k-2}(D_T)}.
\end{align}
Denoting
\[
R_j := \norm{\bR_j}{H^{k,2k}(D_T)},
\quad
m_j := \norm{\bm_j}{H^{k,2k}(D_T)},
\quad
m_{j,0} := |\bm_j|_{H^{k,2k}(D_T)},
\]
in order to estimate each term in the sum on the right hand side of~\cref{eq:ml mlp} 
we use~\cref{eq:rl Rl} and
invoke \cref{lem:huge} with~$\bv=\bm_j$ and~$\bw=\bR_j$, noting that
\[
\RR(\bm_j) = \br_j = (L\partial_t-C_e\Delta)\bR_j,
\]
to obtain
\begin{align}\label{eq:RmR}
\begin{split}
R_{j+1}
&\lesssim
\norm{\br_{j+1}}{H^{k-1,2k-2}(D_T)}\\
&
=
\norm{\RR(\bm_{j+1})}{H^{k-1,2k-2}(D_T)}
=
\norm{\RR(\bm_{j}-\bR_{j})}{H^{k-1,2k-2}(D_T)}
\\
&\lesssim
R_j \norm{\bm_j-\bm^0(x_0)}{H^{k-1,2k-2}(D_T)}
+
\norm{(1-|\bm_j|^2)\Delta\bR_j}{H^{k-1,2k-2}(D_T)}
\\
&\quad
+
%m_{j,0}R_j + m_{j,0}m_jR_j + R_j^2 + m_jR_j^2 + R_j^3.
R_jm_{j,0}(1 + m_j)
+
R_j^2 (1 + m_j)
+
R_j^3.
\end{split}
\end{align}
For the first term on the right hand side of~\cref{eq:RmR}
we note that $\bm_\ell(0,x_0)-\bm^0(x_0)=0$, and hence
\cref{lem:sobolev}~(i) yields (since $k\geq 3$)
\begin{align*}
 |\bm_j(t,x)-\bm^0(x_0)|&\leq ({\rm
diam}(D)^2+T^2)^{1/2}\norm{(\partial_t,\nabla)\bm_j}{L^\infty(D_T)}\\
 &\lesssim \norm{(\partial_t,\nabla)\bm_j}{H^{k-1,2k-2}(D_T)}
 \lesssim |\bm_j|_{H^{k,2k}(D_T)}.
\end{align*}
This implies
\begin{align}\label{eq:nest5}
\norm{\bm_j-\bm^0(x_0)}{H^{k-1,2k-2}(D_T)}
&\leq
\norm{\bm_j-\bm^0(x_0)}{L^2(D_T)}+|\bm_j|_{H^{k-1,2k-2}(D_T)}
\lesssim 
m_{j,0}.
\end{align}
For the second term on the right hand side of~\cref{eq:RmR}, we first observe
that since
\begin{equation}\label{eq:mj m0}
\bm_j
=
\bm_0
-
\sum_{i=0}^{j-1} \bR_i
\quad\text{and}\quad
|\bm_0(t,\cdot)| = |\bm^0| = 1,
\end{equation}
there holds
\[
|m_j|^2
=
|m_0|^2
-
2\bm_0\cdot\sum_{i=0}^{j-1}\bR_i
+
\big|\sum_{i=0}^{j-1}\bR_i\big|^2
\]
so that
\[
1-|\bm_j|^2
=
2\bm_0\cdot\sum_{i=0}^{j-1}\bR_i
-
\big|\sum_{i=0}^{j-1}\bR_i\big|^2.
\]
Thus, with the help of \cref{lem:regnorm}, we obtain
\begin{align}\label{eq:nest1}
\norm{(1-|\bm_j|^2)\Delta\bR_j}{H^{k-1,2k-2}(D_T)}
&\lesssim 
R_j
\Big(
m_0 \sum_{i=0}^{j-1} R_i
+
|\sum_{i=0}^{j-1} R_i|^2
\Big).
\end{align}
Altogether,~\cref{eq:RmR}--\cref{eq:nest1} imply
\begin{align}\label{eq:estt}
R_{j+1}
&\leq 
\widetilde C R_j
\Big(
(m_{j,0} + R_j)(1+m_j)
+ R_j^2
+ m_0\sum_{i=0}^{j-1}R_i +\big|\sum_{i=0}^{j-1}R_i\big|^2
\Big)
\notag\\
&=:
\widetilde CQ_j R_j,
\end{align}
for some constant $\widetilde C>0$, where~$Q_j$ is the sum of all the terms in
the brackets.
We will show that for all $q\in(0,1)$ there exists $\eps>0$ such that
$|\bm^0|_{H^{2k}(D)}\leq \eps$
implies
\begin{equation}\label{eq:CQj}
\widetilde CQ_j \le q\quad\text{for all }j\in\N_0.
\end{equation}
Given $q\in(0,1)$ (and with the constants $C_{\RR}$ from \cref{lem:residual}~(i),
and $C_{\br}$ from~\cref{eq:rl Rl}), we define $C_{\br\RR}:=C_{\br} C_{\RR}(3 + 2|D|+|D|^{1/2})$ and choose $0<\eps<1$ sufficiently small such that
\begin{align}\label{eq:q con}
\epsilon
\Big(
1 &+ C_{\br\RR} + \frac{C_{\br\RR}}{1-q}
\Big)
\Big(
1+\epsilon+|D|^{1/2}+\frac{C_{\br\RR}}{1-q}
\Big)
+
(C_{\br\RR}\epsilon)^2
\notag\\
&+
\Big(
\epsilon+|D|^{1/2}
\Big)
\frac{C_{\br\RR}\epsilon}{1-q}
+
\Big(
\frac{C_{\br\RR}\epsilon}{1-q}
\Big)^2
\le
{\widetilde C}^{-1}q.
\end{align}
This allows us to prove~\cref{eq:CQj} by induction. By assumption, $|\bm^0|_{H^{2k}(D)}$ is sufficiently small such that
\[
m_{0,0}=|\bm^0|_{H^{2k}(D)} \le \epsilon %\frac{\epsilon}{2}
\]
and
\[
\norm{\bm_0-\bc}{H^{k,2k}(D)}=\norm{\bm^0-\bc}{H^{2k}(D)}\leq C_{\rm pc}|\bm^0|_{H^{2k}(D)} \le \epsilon
\]
with $\bc:=|D|^{-1}\int_D \bm^0\in \R^3$ where the Poincar\'e constant $C_{pc}>0$ depends only on $D$. By definition, we have $|\bc|=1$ and hence
\[
m_0 
\le 
\norm{\bm^0-\bc}{H^{k,2k}(D)}
+
\norm{\bc}{H^{k,2k}(D)}
\le 
\epsilon + |D|^{1/2}.
\]
Moreover, since~$\RR(\bc)=0$ we have, noting~\cref{eq:rl Rl},
\begin{align*}
R_0
&\le
C_{\br} 
\norm{\RR(\bm_0)}{H^{k-1,2k-2}(D_T)}
=
C_{\br} 
\norm{\RR(\bm_0)-\RR(\bc)}{H^{k-1,2k-2}(D_T)}
\\
&\le
C_{\br} C_{\RR}(1+m_0^2+\norm{\bc}{H^{k,2k}(D_T)}^2)
\norm{\bm_0-\bc}{H^{k,2k}(D_T)}\\
&\le
C_{\br} C_{\RR}(3 + 2|D|+|D|^{1/2}) \epsilon= C_{\br\RR}\eps.
\end{align*}
Hence
\begin{align*}
Q_0
&=
(m_{0,0} + R_0)(1+m_0)
+ R_0^2
\le
\epsilon (1+C_{\br\RR}) (1+|D|^{1/2}+\epsilon)
+ (C_{\br\RR} \epsilon)^2.
\end{align*}
Our choice of $\eps$ guarantees $\widetilde C Q_0\leq q$. To conclude the induction, assume that~$\widetilde CQ_i \le q$ for all
$i=0,\ldots,j-1$. Then the induction assumption and~\cref{eq:estt} give
\begin{equation}\label{eq:Rj qj}
R_j \le qR_{j-1} \le \cdots \le q^j R_0 \le q^j C_{\br\RR}\epsilon,
\end{equation}
which implies
\[
\sum_{i=0}^{j-1} R_i 
\le \sum_{i=0}^{j-1} q^i R_0
\le \frac{C_{\br\RR}\epsilon}{1-q}.
\]
Hence~\cref{eq:mj m0} proves
\[
m_{j,0} \le m_{0,0} + \sum_{i=0}^{j-1} R_i
\le 
\epsilon
\left(1+\frac{C_{\br\RR}}{1-q}\right)
\]
as well as
\begin{equation}\label{eq:mj m02}
m_{j} \le m_{0} + \sum_{i=0}^{j-1} R_i
\le 
\epsilon
+
|D|^{1/2}
+
\frac{C_{\br\RR}\epsilon}{1-q}.
\end{equation}
It then follows from the definition of~$Q_j$ and $\eps>0$ that~\cref{eq:CQj} holds for
all~$j$. This concludes the induction and proves~\cref{eq:CQj} for all $j\in\N_0$.

We now prove that $\{\bm_{\ell}\}$ is a Cauchy sequence. It follows
from~\cref{eq:ml mlp}, \cref{eq:Rj qj} that
\[
\norm{\bm_{\ell}-\bm_{\ell'}}{H^{k,2k}(D_T)}
\le
\sum_{j=\ell'-1}^{\ell-2} q^{j+1}R_0
\le
\frac{C_{\br\RR}\epsilon}{1-q} q^{\ell'} 
\to 0 \quad\text{as } \ell'\to\infty.
\]
Therefore, $\{\bm_\ell\}$ converges to some $\bm\in H^{k,2k}(D_T)$ which
satisfies, by passing to the limit in the first inequality in~\cref{eq:mj m02},
\begin{align}\label{eq:haha}
\norm{\bm}{H^{k,2k}(D_T)}
&\leq 
\norm{\bm_0}{H^{k,2k}(D_T)}+\sum_{j=0}^{\infty}\norm{\bR_j}{H^{k,2k}(D_T)}
\lesssim 
\norm{\bm_0}{H^{k,2k}(D_T)}
+
\frac{R_0}{1-q}.
\end{align}

It remains to prove that $\RR(\bm)=0$, which can easily be seen from the 
continuity of $\RR$ (see \cref{lem:residual}) and the definition
of~$\bR_{\ell}$:
\begin{align*}
\norm{ \RR(\bm)}{H^{k-1,2k-2}(D_T)}&= \lim_{\ell\to\infty}\norm{ \RR(\bm_\ell)}{H^{k-1,2k-2}(D_T)} \\
&= \lim_{\ell\to\infty}\norm{L\partial_t\bR_\ell -C_e\Delta\bR_\ell}{H^{k-1,2k-2}(D_T)}\lesssim \lim_{\ell\to\infty}\norm{\bR_\ell}{H^{k,2k}(D_T)}\\
&\lesssim \lim_{\ell\to\infty}q^\ell=0.
\end{align*}
As argued at the beginning of this proof, this shows that $\bm|_{D_T}$ is a strong
solution of~\cref{eq:llg}.

Finally, to show~\cref{eq:normsmooth} we note that~\cref{eq:rl Rl}, the continuity
of $\RR$, and the fact that $\RR(0)=0$ yield
\begin{align*}
	R_0\lesssim \norm{\br_0}{H^{k-1,2k-2}(D_T)}=\norm{\RR(\bm_0)-\RR(0)}{H^{k-1,2k-2}(D_T)}\lesssim 
	\norm{\bm_0}{H^{k,2k}(D_T)}.
\end{align*}
Hence~\cref{eq:normsmooth} follows from~\cref{eq:haha}, completing the proof of
the theorem.

\bibliographystyle{plain}
\bibliography{literature}

\end{document}